\documentclass{amsproc}
\usepackage{amsfonts}
\usepackage{amssymb}
\usepackage[english]{babel}
\usepackage[T1]{fontenc}
\usepackage[latin1]{inputenc}
\usepackage[all,dvips]{xy}

\newcommand{\pn}{\par\noindent}

\newcommand{\rad}{\mathrm{rad}\,}
\newcommand{\soc}{\mathrm{soc}\,}
\newcommand{\tp}{\mathrm{top}\,}

\newcommand{\DTr}{\mathrm{DTr}\,}
\newcommand{\DHom}{\mathrm{DHom}\,}
\newcommand{\Hom}{\mathrm{Hom}\,}
\newcommand{\ind}{\mathrm{ind}\,}
\newcommand{\md}{\mathrm{mod}\,}

\newcommand{\pd}{\mathrm{pdim}\,}
\newcommand{\id}{\mathrm{idim}\,}

\newcommand{\LUR}{\mathcal{L}_A \cup \mathcal{R}_A}
\newtheorem{Theorem}{Theorem}[section]

\newtheorem{MainTheorem}{Theorem} {}

\newenvironment{dem}{\noindent {\bf Proof.}}{\hfill $\Box$\\}

\newtheorem{lem}[Theorem]{Lemma}
\newtheorem{cor}[Theorem]{Corollary}

\newtheorem{reNN}{Remark} {}

\begin{document}
\title{Toupie algebra, some examples of laura algebras}

\author[D.~Castonguay - J. Dionne - F.~Huard - M.~Lanzilotta ]
{Diane Castonguay, Julie Dionne, Fran\c{c}ois Huard, Marcelo Lanzilotta
}

\address{\pn Diane Castonguay; Instituto de inform\'atica,
Universidade Federal de Goi\'as, Goi\^ania, Brasil, code postal.}
\email{diane@inf.ufg.br}
\address{\pn Julie Dionne; D\'epartement de Math\'ematiques,
Universit\'e de Sherbrooke, Sherbrooke, Qu\'ebec,
Canada, J1K 2R1.}
\email{dionnej@gmail.com}
\address{\pn Fran\c cois Huard; Department of mathematics,
Bishop's University, Sherbrooke, Qu\'ebec, Canada,  J1M0C8.}
\email{fhuard@ubishops.ca}
\address{\pn Marcelo Lanzilotta; Centro de Matem\'atica (CMAT), Igu\'a 4225,
Universidad de la Rep\'ublica, CP 11400, Montevideo, Uruguay.}
\email{marclan@cmat.edu.uy}

\thanks{
The first and last authors gratefully acknowledges partial support from the Universit\'e de Sherbrooke and University of Bishop's. The first author also gratefully acknowledges partial support from CNPq of Brazil.}

\begin{abstract}
We consider the bound quiver algebras whose ordinary quiver is that
of a canonical algebra. We determine which of those algebras are
hereditary, tilted, quasitilted, weakly shod or laura algebras.
\end{abstract}

\maketitle

\section{Introduction.}

The introduction of quasitilted algebra in the early 90's by Happel,
Ringel and Smal\o \ give rise to the study of two classes
of modules named the right and the left part of the module category.
In \cite{HRS1}, they proved that an algebra is quasitilted if and
only if all indecomposable projective modules lie in the left part
(and equivalently, that all indecomposable injective modules lie in the right part).
Moreover, in this case all indecomposables modules must lie in one of those part.


During the last two decades, the structure of quasitilted algebras and
their module categories has been thoroughly studied and, by a consequence,
the knowledge of the structure of the left and right part of the module category
of an arbitrary algebra also.
Those works leads to various generalisations of classes of algebras.
In this work, we will study examples of the following subclasses of algebras,
each of which is properly contained in the following one:
Hereditary algebras, Tilted algebras, Quasitilted algebras,
Weakly shod algebras, and Laura algebras.

Let $k$ be a commutative field.  A bound quiver algebra $A=kQ/I$ is
called a \textbf{toupie} algebra if $Q$ is a non linear quiver with a
unique source $0$, a unique sink $\infty$ and for any other vertex
$x$ in $Q$, there is exactly one arrow starting at $x$ and exactly
one arrow ending at $x$. Paths going from $0$ to $\infty$ are called
branches of $Q$. Remark that those quiver correspond with
those of canonical algebras introduce in 1984 by Ringel in \cite{R} or are of type $\mathbb{A}_n$.
Our goal is to determine to which of the above classes a given toupie algebra belongs
by looking at its bound quiver. The motivation for this classification of toupie
algebra came from a result of Assem and Coelho \cite{AC} saying that, for the above classes,
for any idempotent $e$ of $A$, the algebra $eAe$ belongs to the same
classes of algebras than $A$, extending \cite{H}, \cite{HRS2} and \cite{KSZ}.
Observe that, for any algebra $A$, there always exist an idempotent $e$ such that $eAe$ is a toupie algebra.

The classification of Toupie algebras into this classes is given as follows:

\begin{MainTheorem}
Let $A=kQ/I$ be a toupie algebra and $m=\dim_k e_0 A e_\infty$ where
$0$ is the unique source of $Q$ and $\infty$ is the unique sink of
$Q$. Let $t$ be the number of branches of $Q$ and assume that $t \geq
2$. Then we have the following:

\begin{itemize}
\item[(H)] The algebra $A$ is hereditary if and only if $m=t$.
\item[(T)] The algebra $A$ is tilted not hereditary if and only if
        \begin{itemize}
        \item[(i)] A is simply connected and $m=1$, or
        \item[(ii)] A is simply connected, $m=t-1$, and $A$ has at most one
        branch of length at least three, or
        \item[(iii)] $m=0$ and there is exactly one relation per branch.
        \end{itemize}

\item[(QT)] The algebra $A$ is quasitilted not tilted if and only if
$A$ is a canonical algebra with $t \geq 3$.


\item[(WS)] The algebra $A$ is weakly shod not quasitilted if and only if $m=0$ and there is at least one branch containing more than one relation.

\item[(L)] The algebra $A$ is Laura not weakly shod if and only if $m=1$ and
exactly one of the branches lies in $I$.

\end{itemize}
\end{MainTheorem}

The linear case, that is when $t=1$, which is equivalent to $Q=\mathbb{A}_n$,
is excluded from our theorem since this has already been characterized:
such an algebra is always representation finite with no cycles in its
Auslander-Reiten quiver.
Hence, it is always weakly shod and it is tilted if and only
if it is quasitilted \cite{HRS2}.  This occurs precisely when there
is exactly one relation on $Q$ \cite{HL}.

In section 2, we give the different notations and definitions that
we intend to use in this paper.  In order to establish our result,
we will treat the simply connected and non simply connected cases
independently.  The non simply connected case is covered in section
3, where it is divided in the following possibilities:

\begin{description}

    \item[Subsection 3.1] All branches are in the ideal.\\
        In this case, $A$ is weakly shod.
        Moreover, we will see that if $A$ is quasi-tilted
        then it is tilted and that this happens only when
        there is exactly one relation per branch.

    \item[Subsection 3.2] None of the branches are in the ideal (and $I\neq 0$).\\
        Then $A$ is not Laura. Thus $A$ cannot be hereditary, nor tilted,
        nor quasitilted, nor weakly shod.

    \item[Subsection 3.3] Some branches are in the ideal.\\
        If $m \geq 2$, $A$ is not Laura.\\
        If $m=1$, then $A$ is not weakly
        shod.  Moreover, $A$ is Laura if and only if
        exactly one of the branches lies in $I$.
        \end{description}

The simply connected case is covered in section 4. Note that in this
case, $t>m > 0$. We then have the following possibilities:

    \begin{description}
    \item[Subsection 4.1] $0< m\leq 2$.\\
        Then $A$ is tilted if $m=1$, and
         $A$ is quasi-tilted (in fact canonical) if it is Laura and $m=2$.
    \item[Subsection 4.2] $m>2$.\\
        If $t > m+1$, then $A$ is not Laura.\\
        If  $t=m+1$ and there is at most one branch of length at
        least three then $A$ is tilted. Otherwise, $A$ is not Laura.
    \end{description}

\section{Notations and definitions.}

We consider only finite dimensional associative algebras over
fields, and all our modules are right finite dimensional modules.

Let $Q$ be a quiver. Given an arrow $\alpha$ in $Q$, we denote by
$s(\alpha)$ the starting vertex of this arrow and by $e(\alpha)$ its
ending vertex.  A \textbf{path} $p=\alpha_1\alpha_2 \cdots \alpha_n$ in
$Q$ is a sequence of distinct arrows $\alpha_i$, $1\leq i \leq n$
such that $e(\alpha_i)=s(\alpha_{i+1})$, for  $1\leq i < n$.

Let $A=kQ/I$ with $Q$ a quiver and $I$ an admissible ideal of $kQ$.
Given a vertex $x$ in $Q$, we denote by $S_x$ the simple module at
$x$, $P_x$ the indecomposable projective module whose top is $S_x$,
and $I_x$ the indecomposable injective module whose socle is $S_x$.
The projective and injective dimensions of an $A$-module $M$ are
denoted respectively by $\pd_AM$ and $\id_AM$. The socle of $M$ is
denoted by $\soc M$, its radical by $\rad M$. We see $A$-module as
representation on $(Q, I)$, see \cite{ASS}. Given a representation on
$(Q, I)$ say $M$ and a vertex $x$ in $Q$, we denote by $M_x$ the $k$-space of the
representation $M$ at the vertex $x$. Similarly, we denote by
$M_\alpha$ the linear transformation from $M_x$ to $M_y$ where
$\alpha$ is an arrow starting at $x$ and ending at $y$. We denote by
$\tau_A$ and $\tau^{-1}_A$ the Auslander-Reiten translations DTr$_A$
and TrD$_A$ respectively over $\md A$.

A finite non linear quiver $Q$ is called a \textbf{toupie} if it has a
unique source $0$, a unique sink $\infty$, and for any other vertex
$x$ in $Q$ there is exactly one arrow starting at $x$ and exactly
one arrow ending at $x$. The general shape of a toupie is thus

$$
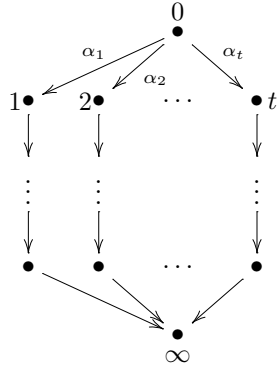
\xymatrix@-0.7pc{%
&& \bullet\save[]+<0pt,8pt>*{0}\restore \ar[lld]_{\alpha_1} \ar[ld]^{\alpha_2} \ar[rd]^{\alpha_t} & \\%
\save[]+<-5pt,0pt>*{1}\restore \bullet \ar[d] & \save[]+<-5pt,0pt>*{2}\restore \bullet \ar[d] & \ldots & \save[]+<6pt,0pt>*{t}\restore \bullet \ar[d] \\%
\vdots \ar[d] & \vdots \ar[d] &  & \vdots \ar[d] \\%
\bullet \ar[rrd] & \bullet \ar[rd] & \ldots & \bullet \ar[ld] \\%
&& \bullet \save[]+<0pt,-8pt>*{\infty}\restore &
}$$

The $t$ distinct paths from $0$ to
$\infty$ are called \textbf{branches} of $Q$. We will denote by $alpha_1, \ldots, \alpha_t$
the $t$ arrows starting at the unique source $0$, by $w_i$ the corresponding branch of $Q$
starting with the arrow $\alpha_i$ and by $i$ the end vertex of $\alpha_i$. Thus, we
have $w_i = 0 \xrightarrow{\alpha_i} i \rightsquigarrow \infty$.
Observe that $i$ can be $\infty$.

Let $k$ be a field, $Q$ a toupie quiver and $I$ an admissible ideal
of $kQ$. We say that $A=kQ/I$ is a {\bf toupie algebra}.

Following Ringel \cite{R}, an algebra $A=kQ/I$ is \textbf{canonical} if
it is a toupie algebra with $t\geq 2$ and $I$ is generated by
\{$w_1+\lambda_i w_2 -w_i$ | $3 \leq i \leq t$\} where the
$\lambda_i$ are pairwise distinct non-zero elements of $k$. Observe
that in the case, that $t=2$ this ideal is the zero ideal.
In particular, we always have that $m=2$ for a canonical algebra.

A \textbf{path} between indecomposable $A$-modules $M$ and $N$ is a sequence
$M=M_0\stackrel{f_1}{\rightarrow}M_1\stackrel{f_2}{\rightarrow}M_2\cdots\stackrel{f_n}{\rightarrow}M_{n}=N$
where each $f_i$ is a non-zero non-isomorphism. An \textbf{IP-path} is
a path from an indecomposable injective $A$-module to an
indecomposable projective $A$-module.

We will gives original or equivalent definition for the studied classes of algebras.
For original definition and history on the rising of these classes, see \cite{ACLST}.

Following \cite{HR}, an algebra is \textbf{tilted} if it is the endomorphism
algebra of a tilting module over an hereditary algebra. This notion was
later generalized by Happel, Reiten and Smal\o \cite{HRS1} with quasi-tilted algebra.
An equivalent definition of this later class can be given by the study
of two classes who plays an important role on representation theory and all
other generalization studied in this work.

Given an algebra $A$, we define two subcategories of $\md A$:

$\mathcal L_A=\{M\in $mod$\,A| $ if there exists a path from an
$A$-module $X$ to $M$, then $\pd X \leq 1$\},

$\mathcal R_A=\{M\in \md A| $ if there exists a path from $M$ to
an $A$-module $X$, then $\id X \leq 1$\}.

In particular, an $A$-module of projective dimension greater than
one cannot lie in $\mathcal L_A$, and an $A$-module of injective
dimension greater than one cannot lie in $\mathcal R_A$.

We say that an algebra $A$ is \textbf{quasi-tilted} if all indecomposable
projectives of $A$ lies in $\mathcal L_A$ (this is equivalent to: all indecomposable
injectives of $A$ lies in $\mathcal R_A$).

We say that $A$ is \textbf{weakly shod} if $\ind A - (\mathcal L_A\cup \mathcal R_A)$ and directed, see \cite{CL03},
and that $A$ is \textbf{laura} if $\ind A - (\mathcal L_A\cup \mathcal R_A)$, see \cite{AC0}.




We are interested in determining when a given toupie algebra is either
hereditary, tilted, quasitilted, weakly shod or Laura.  Clearly, a
toupie algebras is hereditary if and only if $m=t$.

Let $Q$ be a toupie quiver and $I$ be an admissible ideal of $kQ$. A
relation $\rho = \sum\limits_{i \in J} \lambda_i w_i \in  I$ is
called {\it minimal\/} if $|J| \geq 2$ and, for every non-empty
proper subset $J' \subset J$, we have $\sum\limits_{i \in J'}
\lambda_i w_i \notin I$. Observe that $J \subset \{1,\, 2,\,
\ldots,\, t\}$ and that $\lambda_i \neq 0$ for all $i \in J$.

Since a toupie algebra has no double by-pass involved in relations,
it follows from \cite{LeMeur} that a toupie algebra $A=kQ/I$
is simply connected if and only if for each pair of distinct branches
$w_i$ and $w_j$ there exists a minimal relation involving both branches.
We will therefore take this last result for our definition of simply
connected algebra.  For a branch $w_i$ in $Q$, we will denote by
$[w_i]$ the set of branches in $Q$ for which there exists a minimal
relation involving $w_i$.  For more details on simply connected
algebras, we refer the reader to \cite{AS}.

\section{The non simply connected case.}

Let $A=kQ/I$ be a toupie non simply connected algebra and $m=\dim_k
e_0 A e_\infty$ where $0$ is the unique source of $A$ and $\infty$
is the unique sink of $Q$. Consider $t$ the number of branches of
$Q$.

\subsection{All branches are in the ideal}

Throughout this subsection, $A=kQ/I$ is a toupie algebra with
$\dim_k e_0 A e_\infty =0$ and $t \geq 2$. Note that this implies that A is not
simply connected. In this case, we will show that $A$ must be weakly
shod. Moreover, we will see that if $A$ is quasi-tilted then it is
tilted and that this happens precisely when there is exactly one
relation by branch. The first lemma shows such an algebra admits no
sincere indecomposable module.

\begin{lem}\label{lem:zero}
Let $A=kQ/I$ be a toupie non simply connected algebra such that
$\dim_k e_0 A e_\infty =0$.  If $M$ is an indecomposable $A$-module
then $M_0 = 0$ or $M_\infty =0$.
\end{lem}

\begin{dem}
Assume on the contrary that $M_0 \neq 0$ and $M_\infty \neq 0$. This
implies that there must be one branch, say $w_1$, such that
$M_\alpha \neq 0$ for each arrow $\alpha$ on the branch $w_1$. Let
$w_1 = \alpha_1 \alpha_2 \ldots \alpha_s$. Since each branch is in
the ideal $I$, there exists $k \in \{1, \ldots, s-1\}$ such that
$\alpha_1 \alpha_2 \ldots \alpha_k \notin I$ and $\alpha_1 \alpha_2
\ldots \alpha_{k+1} \in I$.

Let $X_k$ be the image of the map $M_{\alpha_k} \circ \ldots \circ
M_{\alpha_1}$. Since $M_{\alpha_{k+1}} \neq 0$, $X_k$ must be
strictly contained in $M_k$ where $k = e(\alpha_k)$. We thus have
$M_k \cong X_k \oplus Y_k$ where $Y_k = X_k^{\perp _{M_k}} \neq 0$.

Let $Y_{k-1} = M_{\alpha_k}^{-1} (Y_k)$ and $X_{k-1} =
Y_{k-1}^{\perp _{M_{k-1}}}$. Then we have that $M_{\alpha_k} =
\begin{pmatrix}
  g_k & 0 \\
  0 & h_k
\end{pmatrix}$
, where $g_k: X_{k-1} \to X_k$ and $h_k: Y_{k-1} \to Y_k$.

We repeat the construction for all $r \in \{k-1, \ldots, 1\}$, that
is, $Y_{r-1} = M_{\alpha_r}^{-1} (Y_r) \neq 0$, $X_{r-1} =
Y_{r-1}^{\perp _{M_{r-1}}}$ and $M_{\alpha_r} =
\begin{pmatrix}
  g_r & 0 \\
  0 & h_r
\end{pmatrix}$
, where $g_r: X_{r-1} \to X_r$ and $h_r: Y_{r-1} \to Y_r$.

We obtain the following:

$$X_0 \oplus Y_0 \xrightarrow{\begin{pmatrix}
  g_1 & 0 \\
  0 & h_1
\end{pmatrix}} X_1 \oplus Y_1 \rightarrow
\ldots \rightarrow X_k \oplus Y_k \xrightarrow{\begin{pmatrix}
  0 & h_{k+1}
\end{pmatrix}} M_{k+1}$$

By definition, $Y_0 = M_{\alpha_1}^{-1} \ldots
M_{\alpha_k}^{-1}(Y_k)$ and $Y_k = X_k^{\perp _{M_k}}$ where $X_k$
is the image of the map $M_{\alpha_k} \circ \ldots \circ
M_{\alpha_1}$. This implies that $Y_0 = 0$. Since $X_k \neq 0$ and
$Y_k \neq 0$, we get a decomposition of $M$. This yields the desired
contradiction.
\end{dem}

As a consequence of the technique used in the above proof, we have
the following result.

\begin{lem}
Let $A=kQ/I$ be a toupie not simply connected algebra with $\dim_k
e_0 A e_\infty =0$. Consider the algebras $B=(1-e_0)A(1-e_0)$ and
$C=(1-e_\infty)A(1-e_\infty)$.
\begin{itemize}
\item[(i)] Let $M$ be an indecomposable $B$-module with $M_\infty \neq
0$, then the socle of the module $M$ is a direct sum of copies of
the simple module $S_\infty$.
\item[(ii)] Let $N$ be an indecomposable $C$-module with $N_0 \neq
0$ then the top of the module $N$ is a direct sum of copies of the
simple module $S_0$.
\end{itemize}
\end{lem}

\begin{proof}
Applying the method of the previous proof, one can easily
decompose $M$ if one of the linear maps $M_\alpha$ is not
injective or $N$ if one of the linear maps $N_\alpha$ is not
surjective.
\end{proof}

The following is now an easy consequence.

\begin{cor}\label{cor:topsoc}
Let $A=kQ/I$ be a toupie not simply connected algebra such that
$\dim_k e_0 A e_\infty =0$. Let $M$ be an indecomposable $A$-module.
If $M_0 \neq 0$ then the top of the module $M$ is a direct sum of
copies of the simple module $S_0$. If $M_\infty \neq 0$ then the
socle of the module $M$ is a direct sum of copies of the simple
module $S_\infty$.
\end{cor}

We now show that $A$ is an extension and a co-extension of weakly
shod algebras.

\begin{lem}
Let $A=kQ/I$ be a toupie non simply connected algebra with $\dim_k
e_0 A e_\infty =0$. Consider the algebras $B=(1-e_0)A(1-e_0)$ and
$C=(1-e_\infty)A(1-e_\infty)$. Then $B$ and $C$ are weakly shod.
\end{lem}

\begin{dem}
We show that $B$ is weakly shod, the other case being dual. Let $I
\rightsquigarrow P$ be an $IP$-path in mod $B$. Note that since
$\infty$ is a sink, there always exists a path from $I_\infty$ to
any indecomposable injective of mod$\,B$. Therefore we can extend
and refine the given path to:

$$I_\infty \to N\rightsquigarrow P$$

\noindent where $N$ is some indecomposable summand of
$I_\infty/ \soc(I_\infty)$. Thus $N_\infty = 0$.

We claim that every successor $X$ of $N$ satisfies $X_\infty = 0$.
Assume there is a map from $N$ to $X$ and $X_\infty \neq 0$. Viewing
$X$ as an indecomposable $A$-module, it follows from corollary
\ref{cor:topsoc} that the socle of $X$ is a direct sum of copies of
the simple module $S_\infty$. Therefore, $N_\infty \neq 0$ which is
a contradiction. Hence, $X_\infty = 0$ and the claim follows by
induction.

Thus each path $N \rightsquigarrow P$ can be viewed as a path over
the algebra $(1-e_\infty)B(1-e_\infty)$. Since each connected
component of this algebra is of finite representation type and
admits no cycle in its Auslander-Reiten quiver, there is a bound on
the length of paths from $N$ to $P$. Thus there is a bound on the
length of $IP$-paths and $B$ is a weakly shod algebra.
\end{dem}

\begin{lem}
Let $A=kQ/I$ be a toupie non simply connected algebra such that
$\dim_k e_0 A e_\infty =0$. Then $A$ is weakly shod.
\end{lem}

\begin{dem}
Let $I \rightsquigarrow P$ be an $IP$-path. We can extend and refine
the path to:

$$I_\infty \to N \rightsquigarrow P_0$$

\noindent where $N$ is some indecomposable summand of $I_\infty/
$soc$\,(I_\infty)$. As in the previous proof, we can show, using
lemma \ref{cor:topsoc}, that the path from $N$ to $P_0$ lies in
mod$\,C$ where $C = (1-e_\infty)A(1-e_\infty)$.

By the previous lemma $C$ is weakly shod so there is a bound on the
length of paths from $N$ to $P_0$ in mod$\,C$. Since each path
$N\rightsquigarrow P_0$ in mod$\,A$ can be viewed as a path in
mod$\,C$ and $I_\infty/$ soc$\,(I_\infty)$ has a finite number of
indecomposable summands we obtain a bound on the length of
$IP-paths$ in mod$\,A$ and $A$ is weakly shod.
\end{dem}

\begin{Theorem}
Let $A=kQ/I$ be a toupie non simply connected algebra such that
$\dim_k e_0 A e_\infty =0$. The following are equivalent:

\begin{itemize}
\item[(a)] The algebra $A$ is tilted
\item[(b)] The algebra $A$ is quasi-tilted
\item[(c)] There is exactly one relation by branch.
\end{itemize}
\end{Theorem}

\begin{dem}
Clearly, (a) implies (b).

To prove that (b) implies (c), assume that there exists more than
one relation on some branch, say $w_1$. Consider the algebra $eAe$
where $e$ is the sum of all the idempotents associated to vertices
on the branch $w_1$. By \cite{HL}, $eAe$ is not quasi-tilted. Thus,
by \cite{AC}, $A$ is not quasi-tilted, a contradiction.

Suppose that there is exactly one relation by branch. Consider
$A=B[M]$ where $B=(1-e_0)A(1-e_0)$ and $M=\rad P_0$. We have that
$B$ is a tilted algebra and that $M$ is an injective decomposable
$B$-module where each summand of $M$ lies on the same slice in the
preinjective component of the Auslander-Reiten quiver of $B$. Thus,
$A$ has a slice in its Auslander-Reiten quiver and, by \cite{HR},
$A$ is tilted. Moreover, the slice in the Auslander-Reiten quiver of
$A$ is a toupie quiver with at most one linear quiver attached to
each branch. Each linear quiver corresponds to the relation on the
respective branch, and has length equal to the length of the
relation on the given branch minus one. The quantity of vertices on
each branch of $A$ corresponds to the same number of vertices of the
corresponding branch and linear quiver on the slice.
\end{dem}

\subsection{No branch is in the ideal}\label{ssec:NONE}

Throughout this subsection, $A=kQ/I$ is a toupie non simply
connected algebra such that $w_i \notin I$ for all $i \in \{1,
\ldots, t\}$, where $t$ is the number of branches of $Q$. We also
assume that $A$ is not hereditary.

In this case, we will show that $A$ is not a Laura algebra thus
cannot be tilted, quasi-tilted or weakly shod.

\begin{lem}
Let $A=kQ/I$ be a toupie non hereditary non simply connected algebra
such that $w_i \notin I$ for all $i \in \{1, \ldots, t\}$, where $t$
is the number of branches of $Q$. Then $A$ is not a Laura algebra.
\end{lem}

\begin{dem}
Since $A$ is neither simply connected nor hereditary, we can assume
that $[w_1] = \{w_1, \ldots, w_r\}$ with $1 < r < t$.

Consider $e = e_0 + e_\infty + \sum\limits_{i=1}^r e_i$. The
quiver of $eAe$ is of the following form:

$$
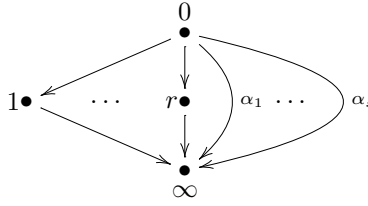
\xymatrix@-0.7pc{%
&& \bullet\save[]+<0pt,8pt>*{0}\restore \ar[lld] \ar[d]
\ar@/^1.5pc/[dd]^{\alpha_1} \ar@/^5pc/[dd]^{\alpha_s} & \\%
\save[]+<-5pt,0pt>*{1}\restore \bullet \ar[rrd] & \ldots & \bullet \save[]+<-5pt,0pt>*{r}\restore \ar[d] & \quad \ldots \\%
&& \bullet \save[]+<0pt,-8pt>*{\infty}\restore &
}$$

\noindent where $1 \leq s < m$. Recall that $m = dim_k e_0 A
e_\infty$ where $0$ is the unique source of $A$ and $\infty$ is the
unique sink of $Q$. Moreover, $m - s < r$ since r corresponds to the
cardinality of $[w_1]$.

Consider the family of indecomposable $eAe$-modules $N_\lambda$.
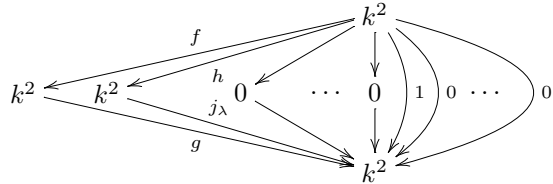
$$\xymatrix@-0.7pc{%
&&&&& k^2 \ar[llllld]_{f} \ar[lllld]^(.6){h} \ar[lld] \ar[d] \ar@/^1pc/[dd]^{1}
\ar@/^2pc/[dd]^{0} \ar@/^5pc/[dd]^{0} & \\%
k^2 \ar[rrrrrd]_{g} & k^2 \ar[rrrrd]^(.4){j_\lambda} && 0 \ar[rrd] &
\save[]+<8pt,0pt>*{\ldots}\restore & 0 \ar[d] & \quad \save[]+<10pt,0pt>*{\ldots}\restore \\%
&&&&& k^2 &
}$$

\noindent where $f= \left[\begin{array}{cc} 1 & 0 \\ 0 & 0\end{array} \right]$;
$g= \left[\begin{array}{cc} 0 & 1 \\ 0 & 1\end{array} \right]$;
$h= \left[\begin{array}{cc} 0 & 1 \\ 0 & 0\end{array} \right]$ and
$j_{\lambda}=\left[\begin{array}{cc} 0 & 1 \\ 0 & \lambda\end{array}\right]$.

Note that $N_{\lambda_1}\not\cong N_{\lambda_2}$ whenever
$\lambda_1\neq \lambda_2$. We will show that $N_\lambda$ has both
projective and injective dimension two. We have the following
projective resolution for $N_\lambda$:

$$ P_\infty^k \to P_1 \oplus P_ 2 \oplus (\bigoplus\limits_{i=3}^{r} P_i^2) \oplus P_\infty^{2s-1} \to P_0^2 \oplus P_1 \oplus P_2 \to N_\lambda$$

\noindent where $k = 2(r-m+s) -1 \geq 2(m -s+1 - m +s) - 1 = 2 -1
=1$. Therefore, $\pd N_\lambda = 2$. Similarly, one can see that
$\id N_\lambda =2$.


Consequently, no $N_\lambda$ lies in $\mathcal L_A \cup \mathcal
R_A$ and thus $eAe$ is not Laura. By \cite{AC}, $A$ is not Laura.
\end{dem}

\subsection{Some branches are in the ideal and some are not.}

Throughout this section, $A=kQ/I$ is a toupie non simply connected
algebra with $m = \dim_k e_0 A e_\infty > 0$, and there exists $w_i
\in I$ for some $i \in \{1, \ldots, t\}$ where $t$ is the number of
branches of $Q$. We can assume, without lost of generality, that
$i=1$.

In this case, we will show that if $m \geq 2$, the algebra $A$
cannot be Laura. When $m=1$, we will show that $A$ cannot be weakly
shod and that $A$ is a Laura algebra if and only if exactly one of
the branches lies in $I$.

\begin{Theorem}
Let $A=kQ/I$ to be a toupie non simply connected algebra such that
$m =
\dim_k e_0 A e_\infty \geq 2$ and $w_1 \in I$.\\

Then $A$ is not a Laura algebra.
\end{Theorem}

\begin{dem}
Consider $e = \sum\limits_{i \in w_1} e_i$, the sum of all
idempotents associated to vertices in the path $w_1$. The quiver of
$eAe$ has the following shape:

$$
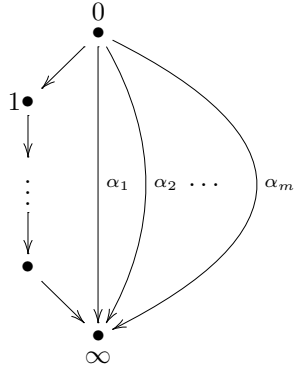
\xymatrix@-0.7pc{%
& \bullet\save[]+<0pt,8pt>*{0}\restore \ar[ld] \ar[dddd]^{\alpha_1}
\ar@/^1.5pc/[dddd]^{\alpha_2} \ar@/^5pc/[dddd]^{\alpha_m} & \\%
\save[]+<-5pt,0pt>*{1}\restore \bullet \ar[d] &  &  \\%
\vdots \ar[d] & & \quad \ldots\\%
\bullet \ar[rd] &&\\%
& \bullet \save[]+<0pt,-8pt>*{\infty}\restore &
}$$

\noindent with some induced zero-relations in $w_1$.

Consider the family of indecomposable $eAe$-modules $N_\lambda$.
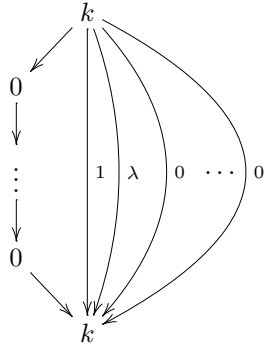
$$\xymatrix@-0.7pc{%
& k \ar[ld] \ar[dddd]^{1}  \ar@/^1pc/[dddd]^{\lambda}
\ar@/^2.5pc/[dddd]^{0} \ar@/^5pc/[dddd]^{0} & \\%
0 \ar[d] &  &  \\%
\vdots \ar[d] & & \quad \quad \ldots\\%
0 \ar[rd] &&\\%
& k  & }$$

Note that $N_{\lambda_1}\not\cong N_{\lambda_2}$ whenever
$\lambda_1\neq \lambda_2$. We will show that these modules have both
projective and injective dimension two. We have the following start
of projective resolution for $N_\lambda$:

$$ P_{i_1} \to P_1 \oplus P_\infty^{m-1} \to P_0 \to N_\lambda$$

\noindent where $i_1$ is such that $0 \rightsquigarrow i_1-1 \notin
I$ and $0 \rightsquigarrow i_1 \in I$. Therefore, $\pd N_\lambda
\geq 2$. In the same way, one can see that $\id N_\lambda \geq 2$.
Therefore $N_\lambda \not\in \mathcal L_A \cup R_A$ and thus $eAe$
is not Laura. By \cite{AC}, $A$ is not Laura.

%

\end{dem}

We now consider the case $m=1$. First, we will show that if $A$ is a
weakly shod algebra, then every $IP$-path must be trivial.

\begin{lem}
Let $A=kQ/I$ be a toupie non simply connected algebra such that
$m=\dim_k e_0 A e_\infty > 0$. If $A$ is a weakly shod algebra, then
every $IP$-path must be trivial.
\end{lem}

\begin{dem}
Assume that $I \rightsquigarrow P$ is a non-trivial $IP$-path. Since
$m > 0$, we have that $(I_\infty)_0 \neq 0$. In particular,
$\Hom_A(P_0, I_\infty) \neq 0$. Since $A$ is a toupie algebra, we
have paths $P \rightsquigarrow P_0$ and $I_\infty \rightsquigarrow
I$. Thus, we get

$$ I_\infty \rightsquigarrow I \rightsquigarrow P \rightsquigarrow P_0 \to I_\infty$$

This gives us a non-trivial cycle in $\ind A$ passing through
injectives and projectives modules. In particular, by \cite{CL03},
$A$ is not weakly shod.
\end{dem}

Thus, in order to show that $A$ is not weakly shod, it suffices to
construct a non-trivial $IP$-path.

\begin{lem}
Let $A=kQ/I$ to be a toupie non simply connected algebra such that
$m=\dim_k e_0 A e_\infty = 1$ and $w_1 \in I$.  Then $A$ is not
weakly shod.
\end{lem}

\begin{dem}
Consider $e = \sum\limits_{i \in w_1} e_i$, the sum of all idempotents
associated to vertices in the path $w_1$. The quiver of $eAe$ is of the following form:

$$
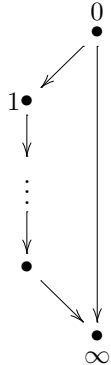
\xymatrix@-0.7pc{%
& \bullet\save[]+<0pt,8pt>*{0}\restore \ar[ld] \ar[dddd]\\%
\save[]+<-5pt,0pt>*{1}\restore \bullet \ar[d]   \\%
\vdots \ar[d] \\%
\bullet \ar[rd] \\%
& \bullet \save[]+<0pt,-8pt>*{\infty}\restore &
}$$

\noindent with some induced zero-relations in $w_1$.

Consider the indecomposable $eAe$-module $N$.
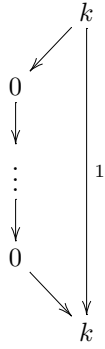
$$\xymatrix@-0.7pc{%
& k \ar[ld] \ar[dddd]^{1} \\%
0 \ar[d] &  \\%
\vdots \ar[d] & \\%
0 \ar[rd] &\\%
& k
}$$

We will show that this module belongs to a non-trivial $IP$-path. It
is sufficient to show that both the projective and injective
dimensions of $N$ are at least two. We have the following start of
projective resolution for $N$:

$$ P_{i_1} \to P_1 \to P_0 \to N$$

\noindent where $i_1$ is such that $0 \rightsquigarrow i_1-1 \notin
I$ and $0 \rightsquigarrow i_1 \in I$. Therefore, $\pd N \geq 2$.
Similarly, one sees that $\id N \geq 2$. This gives us the desired
non trivial $IP$-path

$$I \rightsquigarrow \DTr N \rightsquigarrow N \rightsquigarrow \DTr {-1} N  \rightsquigarrow P$$

By the above lemma, $eAe$ is not weakly shod and thus by \cite{AC}, $A$ is not weakly shod.
\end{dem}

It remains to show that if $m=1$, then $A$ is Laura precisely when
exactly one branch lies in $I$.

\begin{lem}
Let $A=kQ/I$ to be a toupie non simply connected algebra such that
$m=\dim_k e_0 A e_\infty = 1$ and $w_1 \in I$. If $A$ is a Laura
algebra then $w_1$ is the unique branch in $I$.
\end{lem}

\begin{dem}
Assume that $w_1$ is not the unique branch in $I$. We can assume
that $w_2 \in I$. For each $i =1, 2$, let $j_i$ be the first vertex
such that $0 \to i \rightsquigarrow j_i \in I$.

Consider $e = e_0 + \sum\limits_{i} e_i$ where the sum is taken over
all idempotents associated to vertices in the subpath of $w_1$ that
begins by the predecessor of $j_1$ and ends at $\infty$ and in the
subpath of $w_2$ that begins by the predecessor of $j_2$ and ends at
$\infty$. The quiver of $eAe$ has the following shape:

$$
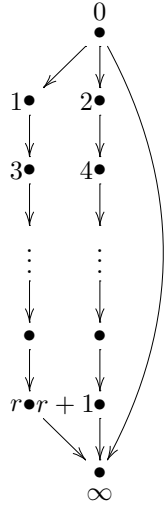
\xymatrix@-0.7pc{%
& \bullet\save[]+<0pt,8pt>*{0}\restore \ar[ld] \ar[d]
\ar@/^2pc/[dddddd] \\%
\save[]+<-5pt,0pt>*{1}\restore \bullet \ar[d] & \save[]+<-5pt,0pt>*{2}\restore \bullet \ar[d]   \\%
\save[]+<-5pt,0pt>*{3}\restore \bullet \ar[d] & \save[]+<-5pt,0pt>*{4}\restore \bullet \ar[d]\\%
\vdots \ar[d] & \vdots \ar[d] \\%
\bullet \ar[d] &\bullet \ar[d]\\%
\save[]+<-5pt,0pt>*{r}\restore \bullet \ar[rd] & \save[]+<-13pt,0pt>*{r+1}\restore \bullet \ar[d]\\%
& \bullet \save[]+<0pt,-8pt>*{\infty}\restore
}$$

\noindent with induced zero-relations of $w_1$ and $w_2$. Observe
that $0 \to 1 \to 3$ is the relation induced by $0 \to 1
\rightsquigarrow j_1$ and that $0 \to 2 \to 4$ is the relation
induced by $0 \to 2 \rightsquigarrow j_2$. Vertices $1$ and $2$ in
$eAe$ correspond to the predecessors of $j_1$ and $j_2$.

Consider the family of indecomposable $eAe$-modules $N_\lambda$.
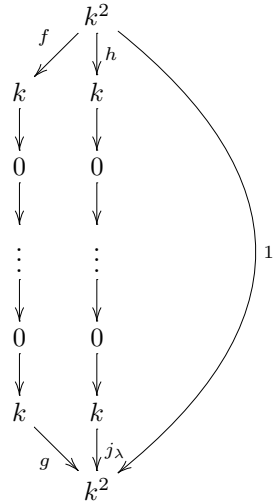
$$\xymatrix@-0.7pc{%
& k^2 \ar[ld]_f \ar[d]^h
\ar@/^5pc/[dddddd]^1 \\%
k \ar[d] & k \ar[d]   \\%
0 \ar[d] & 0 \ar[d]   \\%
\vdots \ar[d] & \vdots \ar[d] \\%
0 \ar[d] & 0 \ar[d]   \\%
k \ar[rd]_g & k \ar[d]^{j_\lambda}\\%
& k^2
}$$

\noindent where $f= \left[\begin{array}{cc} 0 \\ 1\end{array}
\right]$; $h= \left[\begin{array}{cc} 1 \\ 0\end{array} \right]$;
$g= \left[\begin{array}{cc} 1 & 1 \end{array} \right]$ and
$j_{\lambda}=\left[\begin{array}{cc} 1 & \lambda
\end{array}\right]$.

Observe that if one or both of the $j_i, i=1,2$ are the vertices
$\infty$, then the vector space $(N_\lambda)_i $is $k \oplus k$ with
the induced maps. Note that $N_{\lambda_1}\not\cong N_{\lambda_2}$
whenever $\lambda_1\neq \lambda_2$.

We will show that both the projective and injective dimensions of
each $N_\lambda$ is greater or equal to two. We have the following
projective resolution for $N_\lambda$.

$$ 0 \to X \oplus P_\infty^2 \to P_1 \oplus P_ 2 \oplus P_\infty^{2} \to P_0^2 \oplus P_r \oplus P_{r+1} \to N_\lambda$$

\noindent where $X$ corresponds to some quotient of $P_3 \oplus
P_4$. Therefore, $\pd N_\lambda \geq 2$. Similarly, one can see that
$\id N_\lambda \geq 2$. Therefore  $N_\lambda \not\in \mathcal L_A
\cup \mathcal R_A$ and thus $eAe$ is not Laura. By \cite{AC}, $A$ is
not Laura.

\end{dem}

Let $A$ be a toupie non simply connected algebra such that
$m=\dim_k e_0 A e_\infty = 1$.
Assuming that $w_1$ is the unique $w_i$ in the ideal $I$,
we will show that $A$ is a Laura algebra.


We will define a finite family of modules witch are not necessarily $A$-modules. Later one,
we will show that every indecomposables modules $A$-module witch are not lying in $\LUR$ are of this
form or modules on the support of $w_1$.

Let $x$ and $y$ be vertices belong in the path $w_1$. If there exists a path from $x$ to $y$, then
we define $D_{xy}$ to be the module such that $D_{xy}(a)$ is $0$ if $a$ belongs to
the path from $x$ to $y$ and $k$ if not, and $D_{xy}(\alpha)$ is the identity on $k$ if it is possible and $0$ if not. If there exist a path from $y$ to $x$, we define $D_{xy}$ to be the module such that $D_{xy}(a)$ is $k^2$ if $a$
belongs to the path from $y$ to $x$ and $k$ if not, and $D_{xy}(\alpha)$ is the inclusion in the first coordinate when $t(\alpha) = y$ or the projection in the second coordinate when $s(\alpha) = x$ or the identity on $k$ elsewhere.

First of all, we show that if an indecomposable module doesn't vanish in $0$ nor in $\infty$ then it should be of the above form.

\begin{lem}\label{lem:lem1}
Let $A$ be a toupie non simply connected algebra such that
$m=\dim_k e_0 A e_\infty = 1$ and $w_1$ is the unique branch in $I$.
Let $M \in \ind A$.

If $0$ and $\infty$ belongs to the support of $M$ then
$M \cong D_{xy}$ for some $x$ and $y$ in the support of $w_1$.
\end{lem}

\begin{dem}
Fist of all, observe that we can choose a presentation of $A = kQ/I$
such that $w_i - w_j \in I$ for all $i, j \neq 1$.

One can see, with the above hypothesis, that we cannot have more than two simples modules in $\tp M$
(similarly for $\soc M$). Let $O$ and $y$ be the vertices corresponding with the simples modules in $\tp M$
and $\infty$ and $x$ be the ones corresponding with those of $\soc M$. Remark that $0$ can be equal to $y$ and so $\infty$ to $x$ and that $x$ and $y$ must be vertices of the path $w_1$.

Consider $e_w = \sum e_i$ where the sum is taken over the vertices of the path $w_1$ and $e = 1 - e_w + e_0 + e_\infty$. Thus the module $N = eMe$ is also an indecomposable module and one can construct an isomorphism between $N$ and $D_{0\infty}$. Moreover, $e_w M e_w$ is a direct sum of two modules and we can easily construct an isomorphism between $e_w M e_w$ and $e_w D_{xy} e_w$. Using those two morphism, one obtain the desired isomorphism.
\end{dem}

We will see that if an indecomposable module have two simples in is socle including $\infty$
(respectively, in is top including $0$), then it doesn't vanish in $0$ nor in $\infty$, and by the above lemma is isomorph to one of the $D_{xy}$.

\begin{lem}\label{lem:lem2}
Let $A$ be a toupie non simply connected algebra such that
$m=\dim_k e_0 A e_\infty = 1$ and $w_1$ is the unique branch in $I$.
Let $M \in \ind A$.

If $M_\infty \neq 0$ and there exist $x\neq \infty$ such that $x$ correspond to a simples modules in $\soc M$ then $M_0 \neq 0$ and $x$ is a vertex of the path $w_1$.

Similarly, if $M_0 \neq 0$ and there exist $y\neq 0$ such that $y$ correspond to a simples modules in $\tp M$ then $M_\infty \neq 0$ and $y$ is a vertex of the path $w_1$.
\end{lem}

\begin{dem}
If $M_0 = 0$ then $M$ is a $B$-module where $B= (1 - e_0)A(1 - e_0)$. Since $M$ is indecomposable and $M_\infty \neq 0$ and $B$ is a tree algebra with only one sink $\infty$, we have that $x$ must be $\infty$, a contradiction.
\end{dem}

\begin{Theorem}\label{teo:Laura}
Let $A$ be a toupie non simply connected algebra such that
$m=\dim_k e_0 A e_\infty = 1$.

If $w_1$ is the unique branch in $I$, then $A$ is a Laura algebra.
\end{Theorem}

\begin{dem}
Consider $e = \sum_{x \in w_1} e_x$ and $W=eAe$.

Let $K_{R} =\{M \in \ind A \text{ such that } M_0 \neq 0 \text{ or } M \in \ind W \text{ which vanish in } \infty \}$
and $K_{L} =\{M \in \ind A \text{ such that } M_\infty \neq 0 \text{ or } M \in \ind W \text{ which vanish in } 0 \}$.
Observe that $P_0$ and $I_\infty$ belongs to $K_R \cap K_L$. We will show that $K_R$ is closed under successors and $K_L$ is closed under predecessors. Since for all indecomposables modules $X$ which not lies in $\LUR$, we have a path $I_\infty \rightsquigarrow X \rightsquigarrow P_0$, see \cite{ACLST}, we conclude that $\ind A - (\LUR) \subseteq K_R \cap K_L$. Since the proof is dual, we will proved that $K_R$ is closed under successors.

Let $f: L \to M$ be a non-zero morphism between indecomposables modules. Suppose that $L \in K_R$.
First, let suppose that $M_\infty \neq 0$ and $f_\infty \neq 0$. Therefore, $L_\infty \neq 0$ and thus $L_0 \neq 0$.
By lemma \ref{lem:lem1}, we have that $L \cong D_{xy}$ for some $x, y$ in $w_1$. Assuming, $L=D_{xy}$, we obtain that $M(w_2) \circ f_0 = f_\infty L(w_2) = f_\infty \neq 0$ and so $f_0 \neq 0$. We conclude that $M_0 \neq 0$ and therefore that $M \in K_R$.

In the case, where $M_\infty \neq 0$ and $f_\infty = 0$, there exist $x \in \soc M$ such that $L_x \neq 0$ and by the lemma \ref{lem:lem2}, we have that $M_0 \neq 0$.

Thus, we can suppose that $M_0 = 0$ and $M_\infty=0$. Therefore, there exist $y \in \to L$ such that $M_y \neq 0$, witch means that $y \neq 0$. If $L_0 \neq 0$, by lemma \ref{lem:lem2}, we have that $L_\infty \neq 0$ and that $y$ is  a vertex of $w_1$. If $L \in \ind W$, then clearly $y$ belongs to $w_1$. Since $M$ vanishes in $0$ and $\infty$ but not in $y$ a vertex of $w_1$, we have that $M \in \ind W$ and thus $M \in K_R$.

We conclude that $\ind A - (\LUR) \subseteq K_R \cup K_L$ and by lemma \ref{lem:lem1}, we have that $K_R \cup K_L \subseteq \ind W \cup \{D_{xy} \text{ such that } x, y \in W_0\}$ which is finite. Therefore, $A$ is a laura algebra.
\end{dem}

\section{Simply connected case.}

Let $A=kQ/I$ be a toupie simply connected algebra with $m=\dim_k e_0
A e_\infty$. Let $t$ be the number of branches of $Q$.

Observe that, since $t\neq 1$ and $A$ is simply connected, then $t
> m > 0$.

\subsection{Dimension one and two}

Throughout this subsection, $A=kQ/I$ is a toupie simply connected
algebra with $m=\dim_k e_0 A e_\infty \leq 2$.

We will see that $A$ is tilted when $m=1$ and that it is
quasi-tilted (in fact canonical) or not Laura when $m=2$.

\begin{lem}
Let $A$ be a toupie simply connected algebra with $m=\dim_k e_0 A
e_\infty = 1$. Then $A$ is tilted.
\end{lem}

\begin{dem}
Consider $A = B[M]$ where $B=(1-e_0)A(1-e_0)$ and $M = \rad P_0$.
Since $A$ is a simply connected toupie algebra and $m=1$, we have
that $B$ is hereditary and $I_\infty \cong P_0$. Thus $M$ is the
indecomposable injective $B$-module in the vertex $\infty$.
Therefore, $A$ is a one point extension of a hereditary algebra by
an indecomposable injective module. It then follows from  \cite{HR}
that $A$ is tilted.
\end{dem}

\begin{lem}
Let $A$ be a toupie simply connected algebra such that $m=$ $\dim_k
e_0 A e_\infty$ $= 2$. If $A$ is a Laura algebra then it is a
canonical algebra.
\end{lem}

\begin{dem}
Suppose that $A$ is not a canonical algebra. Since $m=2$ and $t>m$
there are at least two branches that are linearly dependant, say
$w_1$ and $w_2$. Thus there exists $\lambda \in k^*$ such that $w_1
- \lambda w_2 \in I$. We can assume, by a simple change of
presentation of $A$, that $w_1 - w_2 \in I$. Consider $e = e_0 + e_1
+ e_2 + e_\infty$.
Then, the algebra $eAe$ is isomorphic to $kQ'/I'$ where $I'$ is the
ideal generated by $\alpha \gamma - \beta \delta$ and the quiver
$Q'$ is the following:

$$
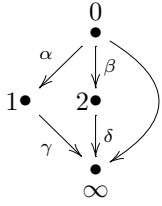
\xymatrix@-0.7pc{%
& \bullet\save[]+<0pt,8pt>*{0}\restore \ar[ld]_\alpha \ar[d]^\beta \ar@/^2pc/[dd] \\%
\save[]+<-5pt,0pt>*{1}\restore \bullet \ar[rd]_\gamma & \bullet \save[]+<-5pt,0pt>*{2}\restore \ar[d]^\delta \\%
                  & \bullet \save[]+<0pt,-8pt>*{\infty}\restore
}$$\\

The algebra $eAe$ is not simply connected, not hereditary  and none
of its branch lies in the ideal $I'$, thus, by subsection
\ref{ssec:NONE}, we have that $eAe$ is not Laura. Therefore, it
follows from \cite{AC} that $A$ is not Laura.
\end{dem}

\subsection{Dimension greater than two}

Thoughout this subsection, $A=kQ/I$ is a toupie simply connected
algebra with $m=\dim_k e_0 A e_\infty \geq 3$.  Let $t$ be the
number of branches of $Q$.

If $t > m+1$, we show that $A$ is not a Laura algebra. In the case
$t=m+1$, if there is at most one branch of length at least three
then $A$ is tilted, otherwise, $A$ is not a Laura algebra.

We will first construct an infinite family of special modules.

\begin{lem}\label{lem:IFam}
Let $A=kQ/I$ be a toupie simply connected algebra such that
$m=\dim_k e_0 A e_\infty \geq 3$.  Consider $e = e_0 + e_\infty +
\sum\limits_{\ i\ \in\ 0^\to} e_i.$

Then there exist $\{ N_\lambda\}_{\lambda \in k^*}$ an infinite
family of non-isomorphic indecomposables $eAe$-modules with
$(N_\lambda)_0 \neq 0$ and $(N_\lambda)_\infty \neq 0$.
\end{lem}

\begin{dem}
Since $A$ is simply connected and $m\geq 3$, we have that $t\geq
4$. Moreover, we can assume that there exists a relation
$\sum\limits_{i=1}^t \lambda_i w_i \in I$ such that $\lambda_1
\neq 0$ and $\lambda_2 \neq 0$. Observe that this relation is not
necessarily a minimal relation. We can easily assume that
$\lambda_1=-1$.

For $\lambda \in k^*$ we define the module $N_\lambda$ as follows:

$$
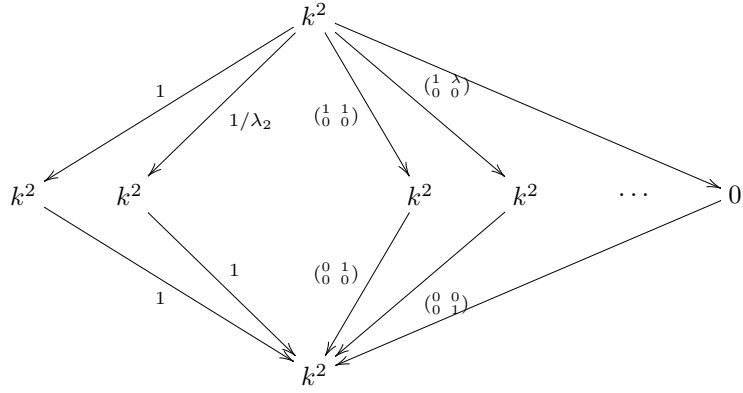
\xymatrix{%
& && k^2 \ar[llldd]_1 \ar[lldd]^{1/\lambda_2} \ar[rdd]_{(^1_0 \ ^1_0)} \ar[rrdd]^{(^1_0 \ ^\lambda_0)} \ar[rrrrdd]&&&&\\%
&&&&&&&\\%
k^2 \ar[rrrdd]_1 & k^2 \ar[rrdd]^1 && & k^2 \ar[ldd]_{(^0_0\ ^1_0)} & k^2 \ar[lldd]^{(^0_0 \ ^0_1)} &\ldots& 0 \ar[lllldd] \\%
&&&&&&&\\%
&&& k^2&&&& }$$\\

It is easy to see that $N_\lambda$ is an indecomposable $eAe$-module
and that $N_{\lambda_1} \cong N_{\lambda_2}$ if and only if
$\lambda_1=\lambda_2$.
\end{dem}

We will now show that the modules $N_\lambda$ lie neither in
$\mathcal L_A$ nor in $\mathcal R_A$, and thus, the algebra is not
Laura.

\begin{lem}
Let $A=kQ/I$ be a toupie simply connected algebra with $m=\dim_k e_0
A e_\infty \geq 3$.  If $t$, the number of branches of $Q$, is
greater than $m+1$, then $A$ is not Laura.
\end{lem}

\begin{dem}
Consider, as above, the algebra $eAe$ where $e = e_0 + e_\infty +
\sum\limits_{\ i\ \in\ 0^\to} e_i$ and $M=\rad _{eAe} P_0$. The
quiver of $eAe$ has the following shape:

$$
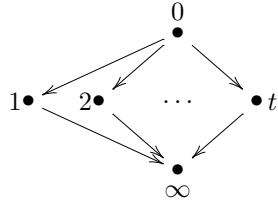
\xymatrix@-0.7pc{%
                  && \bullet\save[]+<0pt,8pt>*{0}\restore \ar[lld] \ar[ld] \ar[rd] & \\%
\save[]+<-5pt,0pt>*{1}\restore \bullet \ar[rrd] & \save[]+<-5pt,0pt>*{2}\restore \bullet \ar[rd] & \ldots & \save[]+<6pt,0pt>*{t}\restore \bullet \ar[ld] \\%
                  && \bullet \save[]+<0pt,-8pt>*{\infty}\restore &
}$$

Our objective is to show that $P_{0{eAe}}$ is not in $L_{eAe}$ and
thus create an $IP$-path. To do this, we will show that the
projective dimension of $\DTr M$ is greater or equal to two. In
order to compute $\DTr M$ in $eAe$, we consider the projective
resolution of $M$.

$$ 0 \to P_\infty ^{(t-m)} \to \bigoplus\limits_{i=1}^t P_i \to M
\to 0 $$

Applying $\DHom ( _ , eAe)$ to this sequence, we get:

$$ 0 \to \DTr M \to I_\infty ^{(t-m)} \to \bigoplus\limits_{i=1}^t
I_i \to \DHom (M, eAe) \to 0$$

Since $M$ is the radical of the projective module of the source $0$
and $(1-e_0)(eAe)(1-e_0)$ is a hereditary algebra, we have that
$\DHom (M, eAe) \cong S_0$, the simple module in $0$.

We thus get that the dimension vector of $\DTr M$ is the
following:

$$
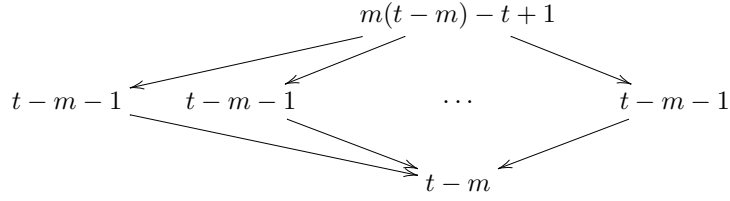
\xymatrix@-0.5pc{%
                  && m(t-m) - t + 1 \ar[lld] \ar[ld] \ar[rd] & \\%
t-m-1 \ar[rrd] & t-m-1 \ar[rd] & \ldots & t-m-1 \ar[ld] \\%
                  && t-m  &
}$$

Consider $r = m(t-m) - t + 1 = (m-1)(t-m-1) > 0$. Since, $eAe$ has
the relations induced by $A$, it is also simply connected and we
have the following sequence:

$$ 0 \to \Omega ^1(\DTr M) \to P_0 ^{r} \to \DTr M \to 0$$

The dimension vector of $\Omega ^1(\DTr M)$ is
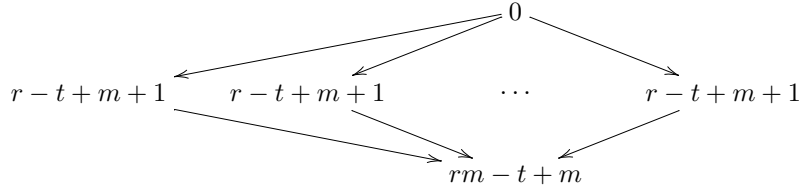
$$\xymatrix@-0.5pc{%
                  && 0 \ar[lld] \ar[ld] \ar[rd] & \\%
r-t+m+1 \ar[rrd] & r-t+m+1 \ar[rd] & \ldots & r-t+m+1 \ar[ld] \\%
                  && rm - t + m  &
}$$

To show that $\Omega ^1(\DTr M)$ is not a projective module, it is
enough to see that $t(r-t+m+1) > rm - t + m$.

We have that %
$$\begin{array}{rcl}
  t(r-t+m+1) &=& ((t-m-1)+(m+1))(r-t+m+1)\\ %
   &=& (t-m-1)(r-t+m+1) + (m+1)(r-t+m+1)\\ %
   &=& rm - t + m - mt + m^2 + m + r + 1 + (t-m-1)(r-t+m+1)\\
\end{array}$$

Moreover,%
$$\begin{array}{rl}%
  & - mt + m^2 + m + r + 1 + (t-m-1)(r-t+m+1)\\%
   &\quad \quad = - mt + m^2 + m + mt - m^2 - t + 1 + 1 + (t-m-1)(r-t+m+1)\\%
   &\quad \quad = -(t - m - 1) + (t-m-1)(r-t+m+1) + 1\\%
   &\quad \quad = (t-(m+1))(r-t+m) +1\\%
   &\quad \quad > 0%
\end{array}$$

Therefore, $\Omega ^1(\DTr M)$ is not a projective module and $\pd
\DTr M \geq 2$. Consequently, $P_0 \notin L_{eAe}$ and we get a
non-trivial $IP$-path $I \rightsquigarrow P_0$ for some
indecomposable injective $I$. Since $\infty$ is a sink, there always
exists a path from $I_\infty$ to any indecomposable injective module
of $eAe$. Thus, we get a non-trivial $IP$-path $I_\infty
\rightsquigarrow P_0$.

On the other hand, by the above lemma, there exists an infinite
family of non-isomorphic indecomposables $eAe$-module $\{ N_\lambda
\}_{\lambda \in k^*}$ and non-zero morphisms $P_0 \to N_\lambda \to
I_\infty$. Combining these morphisms with the $IP$-path, we obtain:

$$I_\infty \rightsquigarrow P_0 \to N_\lambda \to I_\infty$$

Then none of the $N_\lambda$ are in $\mathcal L \cup \mathcal R$, see \cite{ACLST}
and thus $eAe$ is not Laura. By \cite{AC}, $A$ is not Laura.
\end{dem}

We also have the following case where $A$ is not Laura.

\begin{lem}
Let $A=kQ/I$ be a toupie simply connected algebra such that
$m=\dim_k e_0 A e_\infty \geq 3$. If $t$, the number of branches of
$Q$, is equal to $m+1$ and there are at least two branches of length
at least three, then $A$ is not Laura.
\end{lem}

\begin{dem}
We can suppose that $w_1$ and $w_2$ have length at most three, that
is $w_1 = 0 \to a \to b \rightsquigarrow \infty$ and $w_2 = 0 \to c
\to d \rightsquigarrow \infty$. Consider $e  = e_0 + e_\infty + e_b
+ e_c + \sum\limits_{\ i\ \in\ 0^\to} e_i$. The quiver of $eAe$ has
the following shape:

$$
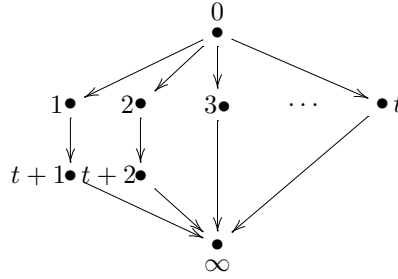
\xymatrix@-0.7pc{%
&& \bullet\save[]+<0pt,8pt>*{0}\restore \ar[lld] \ar[ld] \ar[d] \ar[rrd] && \\%
\save[]+<-5pt,0pt>*{1}\restore \bullet \ar[d] & \save[]+<-5pt,0pt>*{2}\restore \bullet \ar[d] & 3 \bullet \ar[dd] & \ldots & \save[]+<6pt,0pt>*{t}\restore \bullet \ar[lldd] \\%
\save[]+<-12pt,0pt>*{t+1}\restore \bullet \ar[rrd] & \save[]+<-12pt,0pt>*{t+2}\restore \bullet \ar[rd] & && \\%
&& \bullet \save[]+<0pt,-8pt>*{\infty}\restore &&
}$$

Our objective is to show that $P_0$ is not in $L_{eAe}$ and thus
create an $IP$-path. Let $M=$ rad $P_0$.  We will show that $\DTr ^3
M$ has projective dimension greater or equal to two. In order to
compute $\DTr ^3 M$ in $eAe$, we need first to compute $\DTr M$ in
$eAe$ and for this, we consider the projective resolution of $M$.
Remember that $t = m+1$.

$$ 0 \to P_\infty \to \bigoplus\limits_{i=1}^t P_i \to M \to 0 $$

Applying $\DHom ( _ , eAe)$ to this sequence, we get:

$$ 0 \to \DTr M \to I_\infty \to \bigoplus\limits_{i=1}^t I_i \to
\DHom (M, eAe) \to 0$$

Since $M$ is the radical of the projective module of the source $0$
and $(1-e_0)(eAe)(1-e_0)$ is a hereditary algebra, we have that
$\DHom (M, eAe) \cong S_0$, the simple module in $0$.

We thus get that the dimension vector of $\DTr M$ is the
following:

$$
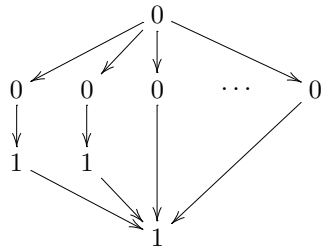
\xymatrix@-0.7pc{%
                  && 0 \ar[lld] \ar[ld] \ar[d] \ar[rrd] && \\%
0 \ar[d] & 0 \ar[d] & 0 \ar[dd] & \ldots & 0 \ar[lldd] \\%
1 \ar[rrd] & 1 \ar[rd] & && \\%
&& 1 && }$$

Since, $\DTr M$ is indecomposable, we have the following
projective resolution:

$$ 0 \to P_\infty \to P_{t+1} \oplus P_{t+2} \to \DTr M \to 0$$

In order to compute $\DTr ^2 M$, we apply $\DHom ( _ , eAe)$ to the
following sequence:

$$ 0 \to \DTr ^2 M \to I_\infty \to I_{t+1} \oplus I_{t+2} \to
\DHom (\DTr M, eAe) \to 0$$

Since $eAe$ has its relations induced by $A$, it is also simply
connected. Therefore, no pair of maps in $P_0$ arriving at vertex
$\infty$ is linearly independent and thus, $\DHom(\DTr M, P_\infty)
= 0$. Moreover,$\DTr M$ is a non projective module over
$(1-e_0)(eAe)(1-e_0)$, a hereditary algebra, thus $\DHom (\DTr M,
eAe) =0$.

We thus get that the dimension vector of $\DTr ^2 M$ is the
following:

$$
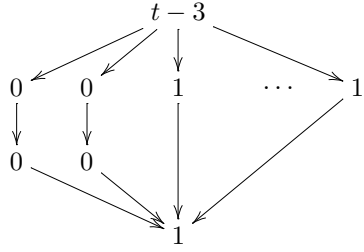
\xymatrix@-0.7pc{%
                  && t-3 \ar[lld] \ar[ld] \ar[d] \ar[rrd] && \\%
0 \ar[d] & 0 \ar[d] & 1 \ar[dd] & \ldots & 1 \ar[lldd] \\%
0 \ar[rrd] & 0 \ar[rd] & && \\%
&& 1 && }$$

Recall that $t \geq 4$. Once again, we obtain the following
projective resolution:

$$ 0 \to (\bigoplus\limits_{i=1}^2 (P_i)^{t-3}) \oplus
(\bigoplus\limits_{i=3}^t (P_i)^{t-4})  \to (P_0)^{t-3} \to \DTr
^2 M \to 0$$

We finally compute $\DTr ^3 M$ by applying $\DHom ( _ , eAe)$ to the
sequence:

$$ 0 \to \DTr ^3 M \to (\bigoplus\limits_{i=1}^2 (I_i)^{t-3})
\oplus (\bigoplus\limits_{i=3}^t (I_i)^{t-4}) \to (I_0)^{t-3} \to
\DHom (\DTr ^2 M, eAe) \to 0$$

Once more, we have that $\DHom (\DTr ^2 M, eAe) =0$ and thus the
dimension vector of $\DTr ^3 M$ is the following:

$$
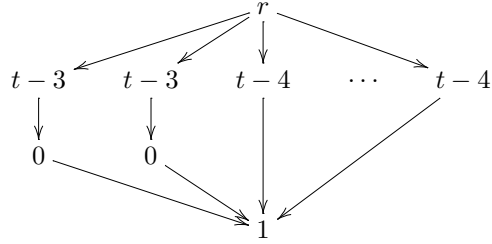
\xymatrix@-0.7pc{%
                  && r \ar[lld] \ar[ld] \ar[d] \ar[rrd] && \\%
t-3 \ar[d] & t-3 \ar[d] & t-4 \ar[dd] & \ldots & t-4 \ar[lldd] \\%
0 \ar[rrd] & 0 \ar[rd] & && \\%
&& 1 && }$$

where $r = t^2 - 5t +5$. Since $\DTr ^3 M$ is indecomposable, we
have the following exact sequence:

$$ 0 \to \Omega ^1(\DTr ^3 M) \to P_0 ^{r} \to \DTr ^3 M \to 0$$

The dimension vector of $\Omega ^1(\DTr ^3 M)$ is
$$
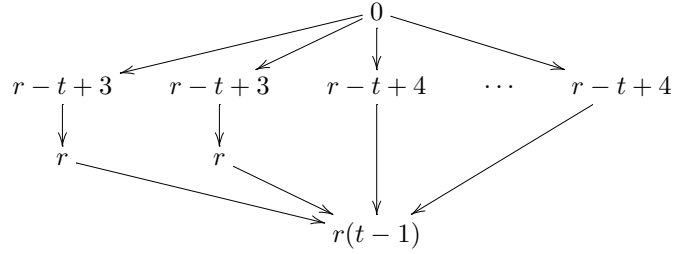
\xymatrix@-0.7pc{%
                  && 0 \ar[lld] \ar[ld] \ar[d] \ar[rrd] && \\%
r-t+3 \ar[d] & r-t+3 \ar[d] & r-t+4 \ar[dd] & \ldots & r-t+4 \ar[lldd] \\%
r \ar[rrd] & r \ar[rd] & && \\%
&& r(t-1) && }$$

To show that $\Omega ^1(\DTr M)$ is not a projective module, it is
enough to see that $2(r-t+3) + (t-2)(r-t+4) + 2 (t-3) > r(t-1)$.

We have that $$\begin{array}{rcl}
  2(r-t+3) + (t-2)(r-t+4) + 2 (t-3) &=& tr - (t^2 -6t+8) )\\ %
   &=& rt - r + t-3 ) \\ %
   &>& r(t-1)\\
\end{array}$$

Therefore, $\Omega ^1(\DTr ^3 M)$ is not a projective module and
$\pd \DTr ^3 M \geq 2$. Consequently, $P_0 \notin L_{eAe}$ and we
get a non-trivial $IP$-path $I \rightsquigarrow P_0$ for some
indecomposable injective $I$. Since $\infty$ is a sink, there always
exists a path from $I_\infty$ to any indecomposable injective of
$eAe$. Thus, we get a non-trivial $IP$-path $I_\infty
\rightsquigarrow P_0$.

On the other hand, by the lemma \ref{lem:IFam}, there exists an
infinite family of non-isomorphic indecomposable $eAe$-module $\{
N_\lambda \}_{\lambda \in k^*}$ and non-zero morphisms $P_0 \to
N_\lambda \to I_\infty$. Combining these morphisms with the
$IP$-path, we obtain:

$$I_\infty \rightsquigarrow P_0 \to N_\lambda \to I_\infty$$

Therefore, none of the $N_\lambda$ are in $\mathcal L \cup
R$, see \cite{ACLST}, and thus $eAe$ is not Laura. By \cite{AC}, $A$ is not Laura.
\end{dem}

In the remaining case, we show that $A$ is tilted.

\begin{lem}
Let $A=kQ/I$ be a toupie simply connected algebra such that
$m=\dim_k e_0 A e_\infty \geq 3$. If $t$, the number of branches of
$Q$, is equal to $m+1$ and there is at most one branch of length at
least three then $A$ is tilted.
\end{lem}

\begin{dem}
The quiver of $A$ has the following shape:

$$
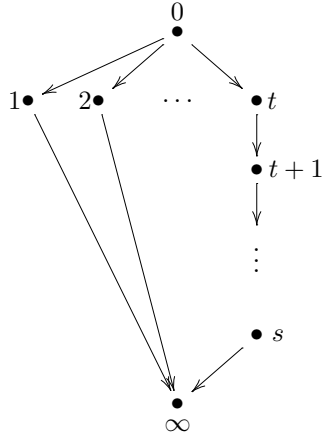
\xymatrix@-0.7pc{%
&& \bullet\save[]+<0pt,8pt>*{0}\restore \ar[lld] \ar[ld] \ar[rd] & \\%
\save[]+<-5pt,0pt>*{1}\restore \bullet \ar[rrdddd] & \save[]+<-5pt,0pt>*{2}\restore \bullet \ar[rdddd] &  \ldots & \save[]+<6pt,0pt>*{t}\restore \bullet \ar[d] \\%
& & & \save[]+<15pt,0pt>*{t+1}\restore \bullet \ar[d] \\%
& & & \vdots \\%
& & & \save[]+<8pt,0pt>*{s}\restore \bullet \ar[ld] \\%
&& \bullet \save[]+<0pt,-8pt>*{\infty}\restore &
}$$

Consider $A=B[M]$ where $B=(1-e_0)A(1-e_0)$ and $M=\rad P_0$. Since
A is simply connected, $B$ is hereditary and $M$ is indecomposable.
We will show that $M$ is a post-projective $B$-module. To do this,
we will show that $\DTr M$ is a projective $B$-module. In order to
compute $\DTr M$, we consider the projective resolution of $M$.
Remember that $t = m+1$.

$$ 0 \to P_\infty \to \bigoplus\limits_{i=1}^t P_i \to M \to 0 $$

Applying $\DHom ( _ , B)$ to this sequence, we get the following
exact sequence:

$$ 0 \to \DTr M \to I_\infty \to \bigoplus\limits_{i=1}^t I_i \to
\DHom (M, B) \to 0$$

Since $M$ is the radical of the projective module of the source $0$
and $B$ is a hereditary algebra, we have that $\DHom (M, B) = 0$.

Since $\DTr M$ is indecomposable, we obtain that $\DTr M \cong
P_{t+1}$. Therefore, $M$ is a post-projective module over the
hereditary algebra $B$, and by \cite{HR}, A is tilted.
\end{dem}

\subsection{Conclusion}%

Given an algebra $A$, we know that if $A$ is tilted, quasitilted,
weakly shod or laura, than any full subcategory of $A$ will also
belong to the given class \cite{AC}.  Let now $A=kQ/I$ be a
triangular algebra.  Then $Q$ is a tree quiver with toupie quivers
glued to it. We can therefore use our main result to show that a
given algebra $A=kQ/I$ is not tilted, quasitilted, weakly shod or
laura. Also, it should be noted that toupie algebras are in general
wild. Consequently ,we now have access to a new class of wild
algebras for which we have, using our characterization as well as
all other known results on tilted, quasitilted, weakly shod and
laura algebras, a great deal of information.


\begin{thebibliography}{15}

\bibitem{AC0} I. Assem; F. U. Coelho,
{\it Two-sided gluings of tilted algebras},
J. Algebra {\bf 269} (2) (2003), 456-479

\bibitem{AC} I. Assem; F. U. Coelho,
{\it Endomorphism rings of projective modules modules over laura algebras},
J. Algebra and its Appl. {\bf 3} 1 (2004), 49-60.

\bibitem{ACLST} I. Assem; F. U. Coelho; M. Lanzilotta; D. Smith; S. Trepode,
{\it Algebras determined by their left and right parts. Algebraic structures and their representations},
Contemp. Math. {\bf 376}, Amer. Math. Soc., Providence, RI (2005), 13-47.

\bibitem{ASS} I. Assem; D. Simson; A. Skowro\'n ski,
{\it Elements of the representation theory of associative algebras. Vol. 1. Techniques of representation theory},
London Mathematical Society Student Texts {\bf 65} (Cambridge University Press, 2006), x+458 pp.

\bibitem{AS} I. Assem; A. Skowro\'nski,
{\it On some class of simply connected algebras.},
Proc. London Math. Soc. {\bf 3}(56) (1988), 417-450.

\bibitem{CL} F. Coelho; M. A. Lanzilotta,
{\it Algebras with small homological dimensions},
Manuscripta math., {\bf 100}(1) (1999), 1-13.

\bibitem{CL03} F. U. Coelho; M. Lanzilotta,
{\it Weakly shod algebras},
J. Algebra {\bf 265}(1) (2003), 379-403.

\bibitem{H} D. Happel,
{\it Triangulated Categories in the Representation Theory of Finite Dimensional Algebras},
London Math. Soc. Lecture Note Series {\bf 119} (Cambridge University Press, 1988), ix + 208 pp.


\bibitem{HR} D. Happel; C. M. Ringel,
{\it Tilted algebras},
Trans. Amer. Math. Soc {\bf 274} (1982), 399-443.

\bibitem{HRS1} D. Happel; I. Reiten; S. O. Smal\o,
{\it Quasitilted algebras, Finite-dimensional algebras and related topics}
(Ottawa, ON, 1992),
NATO Adv. Sci. Inst. Ser. C Math. Phys. Sci., {\bf 424}, Kluwer Acad. Publ., Dordrecht, (1994) 163-181.

\bibitem{HRS2} D. Happel; I. Reiten; S. O. Smal\o,
{\it Tilting in abelian categories and quasitilted algebras},
Mem. Amer. Math. Soc. {\bf 120} (1996), no. 575, viii+ 88 pp.

\bibitem{HL} F. Huard; S. Liu, {\it Tilted string algebras},
J. Pure and Applied Algebra, {\bf 153} (2) (2000), 151-164.

\bibitem{KSZ} M. Kleiner; A. Skowro\'n ski; D. Zacharia,
{\it On endomorphism algebras with small homological dimensions},
J. Math. Soc. Japan {\bf 54} (2002) 621-648.

\bibitem{LeMeur} P. Le Meur,
{\it The universal cover of an algebra without double bypass},
J. Algebra {\bf 312} (2007), no. 1, 330-353.

\bibitem{R} C. M. Ringel,
{\it Tame algebras and integral quadratic forms},
Lecture Notes in Mathematics, {\bf 1099} Springer-Verlag, Berlin (1984) xiii+376 pp.

\end{thebibliography}
\end{document}